\documentclass[11pt,a4paper]{article}

\usepackage{amssymb, amsmath,color, bm}

\usepackage{hyperref}

\newcommand{\hide}[1]{}

\newcommand \A {\mathbf{A}}

\newcommand \N {\mathbb{N}}

\newcommand \R {\mathbf{R}}
\newcommand \D {\mathbf{D}}

\newcommand \cF {\mathcal{F}}

\newcommand \cP {\mathcal{P}}

\newcommand \sign{\mathrm{sign}}

\newcommand \Der{{\rm Der}}

\newcommand \SIGN {{\rm SIGN}}

\newtheorem{defn}{Definition}

\newtheorem{lemma}[defn]{Lemma}
\newtheorem{proposition}[defn]{Proposition}
\newtheorem{theorem}[defn]{Theorem}
\newtheorem{corollary}[defn]{Corollary}
\newtheorem{remark}[defn]{Remark}
\newtheorem{notation}[defn]{Notation}

\newenvironment{proof}[1]{
  \trivlist \item[\hskip \labelsep{\it #1}]}{\hfill\mbox{$\square$}
  \endtrivlist}

\title{Elementary recursive quantifier elimination based on \linebreak Thom encoding and sign determination}

\author{Daniel Perrucci$^{\flat}$\thanks{Partially supported by the Argentinian  grants 
PIP 2014-2016 11220130100527CO CONICET and UBACYT 20020120100133.}  \quad  Marie-Fran\c{c}oise Roy$^{\sharp}$ \\[5mm]
{\small ${\flat}$ Departamento de Matem\'atica, FCEN, Universidad de Buenos Aires and IMAS UBA-CONICET,}\\
{\small Ciudad Universitaria, 1428 Buenos Aires, Argentina}\\ 
{\small ${\sharp}$ IRMAR (UMR CNRS 6625), Universit\'e de Rennes 1,} \\
{\small Campus de Beaulieu, 35042 Rennes, cedex,  France}}

\begin{document}

\maketitle

\begin{abstract}
We describe a new quantifier elimination algorithm for real closed fields based on Thom encoding and sign 
determination. The complexity of this algorithm is elementary recursive and its proof of correctness is completely
algebraic. In particular, the notion of connected components of semialgebraic sets is not used. 

\end{abstract}

\bigskip


\medskip


\section{Introduction}

The first proofs of quantifier elimination for real closed fields by 
Tarski, Seidenberg, Cohen or H\"ormander (\cite{Tarski51,Seidenberg54,Coh,Hor})
were all providing primitive recursive algorithms. 

The situation changed with the Cylindrical Algebraic Decomposition  method (\cite{Col}) 
and elementary recursive algorithms where obtained (see also \cite{Loj, Mo}).
This method produces a set of sampling points meeting every connected component defined by a sign condition 
on a family of polynomials. 
Cylindrical Algebraic Decomposition,
being based on repeated projections, is in fact doubly exponential in the number of variables 
(see for example \cite[Chapter 11]{BPRbook}).

Single exponential degree bounds, using the critical point method to project in one step a block of variables, 
have been obtained for the existential theory over the reals.
The critical point method also gives a quantifier elimination algorithm which is doubly exponential in 
the number of blocks (\cite{DH, Gri88,GV,R92,BPRQE,BPRbook}).

For all these elementary recursive methods, the proofs of 
correctness of the algorithms are based on geometric properties of semialgebraic sets, such as 
the fact that they have a finite number of connected components. 
They are also valid for general real closed fields, 
where the notion of semialgebraic connectedness has to be used.

Our aim in this paper is
to provide an elementary recursive algorithm for quantifier elimination over real closed fields
(Theorem \ref{quantifier_elimination})
with the particularity that 
its proof of correctness 
is entirely based on algebra and does not involve the notion of connected components 
of semialgebraic sets (see details in Remark \ref{rem:prop_based_on_alg}, Remark \ref{elimalgebraic} and Remark \ref{purelyalgebraic}).

The development of such algebraic proofs is very important in the field 
of constructive algebra. For instance,
the elimination of one variable step 
of the algorithm we present here is,
in the special case of monic polynomials,
 a key step in the construction o
 algebraic identities with elementary recursive degree bounds for the 
Positivstellensatz and Hilbert 17'th problem in  \cite{zerelem}.

Another motivation for the present work is to provide an elementary recursive algorithm
for quantifier elimination over real closed fields, suitable 
for being formally checked by a proof assistant such as \texttt{Coq} \cite{coq} 
using the algebraic nature of its correctness proof. 
Indeed, because of the algebraic nature of its correctness proof, the original proof of Tarski's quantifier elimination 
\cite{Tarski51},
as presented in \cite[Chapter 2]{BPRbook} has already been checked using
\texttt{Coq}  in \cite{Assia_Cyril}.

We start with some notation.

Let $\R$ be a  real closed field.  For $\alpha \in \R$, its \emph{sign} is as usual defined as follows:
$$
{\rm sign} (\alpha) =  \left \{ \begin{array}{rl}
-1 & {\rm if } \ \alpha < 0,\\    
0 & {\rm if } \ \alpha = 0,\\
1 & {\rm if } \ \alpha > 0.\\
       \end{array} \right.
$$
Given a family of polynomials $\cF \subset \R[x_1, \dots, x_k]$, 
a \emph{sign condition} on  $\cF$ is an element 
$\tau$ of $\{-1, 0, 1\}^{\cF}$.
We use the notation 
$$
{\rm sign}(\cF)=\tau
$$ 
to mean
$$
\bigwedge_{Q\in \cF} \left(  {\rm sign}(Q)=\tau(Q)\right).
$$
The \emph{realization} of a sign condition $\tau$ on $\cF$ is defined  as
$$
{\rm Real}(\tau, \R)=\{\upsilon \in \R^k \mid {\rm sign}(\cF(\upsilon))=\tau\}.
$$
If ${\rm Real}(\tau, \R) \ne \emptyset$, we say that $\tau$ is \emph{realizable}. Finally, 
we note by $\SIGN(\cF)$ the set of realizable sign conditions on $\cF$.

For $p \in \mathbb{Z}$, $p\ge 0$, we denote by ${\rm bit}(p)$ the number of binary digits needed to represent $p$. This 
is to say
$$
{\rm bit}(p) = \left\{
\begin{array}{rcl}
1 & \hbox{if} & p = 0, \cr
k &  \hbox{if} & p \ge 1 \hbox{ and } 2^{k-1} \le p < 2^{k} \hbox{ with } k \in \mathbb{Z}.
\end{array}
\right.
$$

Let $\D \subset \R$ be a subring. 
In this paper, given a finite family of polynomials $\cF \subset \D[x_1, \dots, x_k]$, 
we will construct for $1 \le i \le k-1$ 
a new explicit family of polynomials ${\rm Elim}_i(\cF) \subset \D[x_1, \dots, x_i]$ 
which 
is suitable for quantifier elimination
on first order formulas with atoms defined by polynomials in $\cF$.

For organization matters, the definition of the family ${\rm Elim}_i(\cF)$ 
is posponed to 
Definition \ref{def:higher_order_eliminating_families} in Section \ref{sec:main}, and   
we include below our main result, which is Theorem \ref{quantifier_elimination}.
This theorem also states \emph{complexity} bounds for the quantifier elimination method we present. 
Roughly speaking, the complexity is the number of operations in $\D$ that the computation takes; this concept will 
be further explained in Section \ref{sec:preliminaries}.

\begin{theorem}\label{quantifier_elimination}  
Let $\cF \subset 
\D[x_1, \dots, x_k]$ be a finite family of polynomials.
Given a first order formula of type 
$$
{\rm Qu}_{i+1} x_{i+1} \dots {\rm Qu}_{k} x_{k} \, \Phi(x_1, \dots, x_k)
$$
with $1 \le i \le k-1$,  ${\rm Qu}_h \in \{\forall, \exists\}$ for $i + 1 \le h \le k$ and $\Phi(x_1, \dots, x_k)$ a quantifier
free formula with atoms defined by polynomials in $\cF$, 
there exists an equivalent quantifier free formula 
$$
\Psi(x_1, \dots, x_i)
$$ 
with atoms in 
${\rm Elim}_i(\cF)$. 
More precisely, there exists $T_\Phi\subset \SIGN({\rm Elim}_i(\cF))$
so that 
$$\begin{array}{rcl}
\Psi(x_1, \dots, x_{i})& = &
\displaystyle{\bigvee_{\tau \in T_{\Phi}}
\left( \sign({\rm Elim}_{i}(\cF)) = \tau \right). }
\end{array}$$

If $\cF \subset \D[x_1, \dots, x_k]$ is formed by $s$ polynomials of degree bounded by $d$, 
then 
$$
\# {\rm Elim}_{i}(\cF) \le s^{2^{k-i}}\max\{2, d\}^{(16^{k-i}-1){\rm bit}(d)},
$$
the degree of the polynomials in ${\rm Elim}_i(\cF)$ 
is bounded by 
$$
4^{\frac{4^{k-i}-1}{3}}d^{4^{k-i}},
$$
and 
the 
complexity of computing the quantifier free formula
$\Psi$
is 
$$
O \Big(
s^{2^{k}}\max\{2, d\}^{{\rm bit}(d)(16^{k} + (k-1) 4^{k+1})}
\Big)
$$
operations in $\D$.
\end{theorem}

This paper is organized as follows. In Section \ref{sec:preliminaries} we 
state some preliminaries on 
complexity, Thom encodings, Tarski queries and Sign determination. In Section \ref{sec:projecting_one}, we develop the main
step of our construction, which is the elimination of one variable. 
Finally, in Section \ref{sec:main}, we prove Theorem \ref{quantifier_elimination}.

\section{Preliminaries} \label{sec:preliminaries}

\subsection{Complexity}

The computations we consider in this paper
perform arithmetic operations 
in a subring $\D$ of a real closed field $\R$.
The notion of {\em complexity} of a computation we
consider is the number of arithmetic operations in $\D$ done during the
described procedure.
We consider that sign evaluation in $\D$  is cost free.
We also consider that accessing, reading and writing
pre-computed objects is cost free. For instance, we can access at any
moment for free to any specific
coefficient of a multivariate polynomial or any specific entry of a matrix.
Also, we do not consider the cost of doing arithmetic operations between 
auxiliar numerical quantities (such as cardinalities of sets). In short, we focus on the operations in 
$\D$, which is the natural ambient for our input.

For the complexity of basic algorithms 
for polynomial operations we refer to \cite[Chapter 8]{BPRbook}.
Also, we use 
Berkowitz Algorithm \cite{Berk} as a division free algorithm to compute the determinant of 
a $p \times p$ matrix with entries in a commutative ring $\A$, 
within $O(p^4)$ operations in $\A$.

\subsection{Thom encodings}

We recall now the Thom encoding of real algebraic numbers  \cite{CR} and explain 
its main properties. We refer to \cite[Section 2.1]{BPRbook} for classical proofs and
to \cite[Section 6.1]{zerelem} for
proofs based on algebraic identities coming from Mixed Taylor Formulas.

\begin{defn}\label{def:Thom}
Let $P(y) = \sum_{0 \le h \le p}\gamma_hy^h \in \R[y]$ with $p\ge 1$ and $\gamma_p \ne 0$. 
We denote $\Der(P)$ the list formed by $P$ and the first $p-1$ derivatives of $P$.

Given a real root $\theta$ of $P$, 
the {\em Thom encoding} of $\theta$ with respect to $P$
is the list of signs of 
$\Der(P')$ evaluated at $\theta$.
\end{defn}

Every real root of $P$ is uniquely determined by its Thom encoding with respect to $P$; 
in the sense that two different 
real roots can not have the same Thom encoding.

For convenience we identify sign conditions on $\Der(P')$ (resp. $\Der(P)$), which are
by definition elements in
$\{1,0,-1\}^{\Der(P')}$ (resp. $\{1,0,-1\}^{\Der(P)}$), 
with elements in $\{-1,0,1\}^{\{1,\ldots,p-1\}}$ (resp. $\{-1,0,1\}^{\{0,\ldots,p-1\}}$).
By convention, for any sign condition
$\eta$ on $\Der(P')$ or $\Der(P)$ we extend its definition with $\eta(p)
= \sign(\gamma_p)$.

It is clear that the multiplicity of a real root of $P$ can be deduced from its Thom encoding. Also, 
Thom encodings can 
be used to order real numbers as follows.

\begin{notation}\label{not_order_thom_encoding}
Let $P(y) = \sum_{0 \le h \le p}\gamma_hy^h \in \R[y]$ with $p\ge 1$ and $\gamma_p \ne 0$. 
For $\eta_1, \eta_2$ 
sign conditions on $\Der(P)$, we use the notation $\eta_1 \prec_P \eta_2$
to indicate that $\eta_1 \ne \eta_2$ and, 
if $q$ is the biggest value of $k$ such that $\eta_1(k) \ne \eta_2(k)$, then
\begin{itemize}
 \item 
$\eta_1(q)
< \eta_2(q)
$ and $\eta_1(q+1) = 1$ or
 \item 
$\eta_1(q) > \eta_2(q)$ and $\eta_1(q+1) = -1$.
\end{itemize}
We use the notation $\eta_1 \preceq_P \eta_2$ to indicate
that either $\eta_1 = \eta_2$ or $\eta_1 \prec_P \eta_2$. 
\end{notation}

It is easy to see that $\preceq_P$ defines a partial order on $\{-1, 0, 1\}^{\Der(P)}$.
In addition, $\preceq_P$ defines a total order on $\SIGN(\Der(P))$. 
Indeed, let $\theta_1, \theta_2 \in \R$, $\eta_1 = \sign(\Der(P)(\theta_1))$ and  
$\eta_2 = \sign(\Der(P)(\theta_2))$ with $\eta_1 \ne \eta_2$, and 
let $q$ be as in Notation \ref{not_order_thom_encoding}. 
Note that since $\eta_1(p) = \eta_2(p) = \sign(\gamma_p)$, then $q < p$.
It is not possible that there exists $k$ such that $q < k < p$ 
and $\eta_1(k) = \eta_2(k) = 0$; otherwise
$\theta_1$ and $\theta_2$ would be roots of 
$P^{(k)}$ with the same Thom encoding with respect to this polynomial,
and therefore $\theta_1 = \theta_2$, which is impossible since 
$\eta_1 \ne \eta_ 2$. In particular, we have then that either 
$\eta_1(q+1) = 1$ or $\eta_1(q+1) = -1$ and therefore it is possible
to order $\eta_1$ and $\eta_2$ according to $\preceq_P$.

\begin{proposition}\label{lemma_order_thom_encoding} 
Let $P(y) =\sum_{0 \le h \le p}\gamma_hy^h \in \R[y]$ with $p\ge 1$ and $\gamma_p \ne 0$
and $\theta_1, \theta_2 \in \R$. If 
$\sign(\Der(P)(\theta_1)) \prec_P \sign (\Der(P)(\theta_2))$ 
then $\theta_1 < \theta_2$.
\end{proposition}

\subsection{Tarski queries}

Let $P, Q \in \R[y]$ with $P \not \equiv 0$.
The {\rm Tarski-query} of $Q$ for
$P$  
is 
$$
\begin{array}{rcl}
{\rm TaQu} (Q;P) & = & \displaystyle \sum_{\theta \in \R \, | \, P(\theta) = 0}
   {\rm sign} (Q(\theta)) \\[7mm]
& = & \# \left\{ \theta \in \R \mid P(\theta) = 0, \,
   Q(\theta) > 0 \right\} - \#  \left\{ \theta  \in \R \mid P(\theta ) = 0, \,
   Q(\theta) < 0 \right\}. 
\end{array}
$$

There are several methods to compute the Tarski-query of $Q$ for $P$.
Here, we describe one which is well adapted to the parametric case.

\begin{defn}[Hermite's Matrix]\label{def:her_mat}
Let  $P, Q \in \R[y]$ 
with $\deg P = p \ge 1$. 
The Hermite's matrix ${\rm Her}(P;Q) \in \R^{p \times p}$
is the matrix defined for $1 \le j_1, j_2 \le p$ by 
$$
{\rm Her}(P;Q)_{j_1, j_2} = 
{\rm Tra}
(Q(y) y^{j_1 + j_2 - 2})
$$
where  ${\rm Tra}(A(y)
)$ is the trace of the linear mapping of multiplication by $A(y)
\in \R[y]$
in the $\R$-vector space $\R[y]/P(y)$.
\end{defn}

\begin{remark}\label{rem:degree_her_mat}
Let  $P(y) = \sum_{0 \le h \le p} \gamma_hy^h, 
Q = \sum_{0 \le h \le q} \gamma'_hy^h \in \R[y]$ 
with $p\ge 1$ and $\gamma_p \ne 0$.

For $j \in \N$ we denote by ${\rm A}_{p,j} \in \mathbb{Z}[c_0, \dots, c_{p-1}]$ the unique 
polynomial 
such that 
$${\rm A}_{p,j}(s_p(y_1, \dots, y_p), \dots, s_1(y_1, \dots ,y_p)) = 
\sum_{1 \le k \le p}y_k^j \in \mathbb{Z}[y_1, \dots, y_p],$$
where 
for $1 \le j \le p$, $s_j(y_1, \dots, y_p)$ is the $j$-th elementary symmetric function evaluated in $y_1, \dots, y_p$. 
Note that  $\deg {\rm A}_{ p, j} = j$ (see \cite[Proof of Theorem 3, Chapter 7]{CLO}).  

Then we have that for  $1 \le j_1, j_2 \le p$, 
$$
 {\rm Her}(P;Q)_{j_1, j_2}
= 
\sum_{0 \le h \le q} \gamma'_h {\rm A}_{p,h + j_1 + j_2 - 2}
\left((-1)^p\frac{\gamma_0}{\gamma_p}, \dots, -\frac{\gamma_{p-1}}{\gamma_p}\right)
$$
(see \cite[Section 4.3]{BPRbook}); therefore 
$$
\gamma_p^{q + 2p - 2}{\rm Her}(P;Q)_{j_1, j_2} 
$$
is a polynomial in the coefficients of $P$ and $Q$ with 
degree $q + 2p - 2$ with respect to the coefficients of $P$
and degree $1$ with respect to the coefficients of $Q$.
\end{remark}

\begin{theorem}[Hermite's Theory (1)]
\label{hermite} Let $P, Q
\in \R[y]$ with $\deg P = p \ge 1$. Then
$$
{\rm Si}({\rm Her}(P;Q))  = {\rm TaQu} (Q;P)
$$
where ${\rm Si}({\rm Her}(P;Q))$ is the signature of the symmetric matrix ${\rm Her}(P;Q)$. 
\end{theorem}

\begin{proof}{Proof:} See \cite[Theorem 4.58]{BPRbook} or \cite[Section 5.1]{zerelem} for a 
proof based on algebraic identities. 
\end{proof}

A nice property of the Hermite's matrix 
is that its signature can always be computed from the sign of its principal minors
(property which is not extensive to general symmetric matrices, or even Hankel matrices, as shown for instance
by the matrices $ \left(\begin{array}{cc}
                  0 & 0 \cr 0 & 0
                 \end{array}\right)$ and $\left(\begin{array}{cc}
                  0 & 0 \cr 0 & 1
                 \end{array}\right)$, having same principal minors and different signatures).

\begin{notation}\label{not:Hermite_minors}
Let $P, Q \in \R[y]$ with $\deg P = p \ge 1$.
For $0 \le j \le p-1$, we denote by
${\rm hmi}_j(P;Q)$  the $(p-j)$-th principal minor of ${\rm Her}(P;Q)$. 
We extend this definition with ${\rm hmi}_{p}(P;Q) = 1$. 
We denote by   
${\rm hmi}(P;Q)$ the list 
$$ 
[{\rm hmi}_{0}(P;Q), \dots, {\rm hmi}_{p-1}(P;Q),1] \subset \R.
$$ 
\hide{
and by ${\rm hmi}_+(P;Q)$ the list ${\rm hmi}(P;Q)$ extended with ${\rm hmi}_{p}(P;Q) =1$.
}
\end{notation}

We also consider the following notation. 

\begin{notation}
\label{notation:testcoeff}
\begin{itemize}
\item For $k \in \N, \varepsilon_k = (-1)^{k(k-1)/2}$.

\item Let $h=h_0,\ldots,h_p$ be a finite list
in $\R$ such that $h_p \not=0$.
We 
denote by  $(d_0,\ldots,d_s)$ the strictly
decreasing sequence  of natural numbers  
defined by 
$\{d_0, \dots, d_s\} = \{ j \ | \ 0 \le j \le p, \, h_j \ne 0\}$.  
We define 
$$ 
{\rm PmV}(h)= 
\sum_{1 \le i \le s, \atop d_{i-1}-d_i \, \hbox{\scriptsize
 {odd}}} 
\varepsilon_{d_{i-1}-d_{i}} 
\sign(h_{d_{i-1}})\sign(h_{d_i}). 
$$
\end{itemize}

\end{notation}

Note that in Notation \ref{notation:testcoeff} it is always the case that $d_0 = p$. 
Also, when all elements of $h$ are non-zero,
${\rm PmV}(h)$
 is the difference
between the number of sign permanencies and
the number of sign changes in
$h_p,\ldots,h_0$.

\begin{theorem}[Hermite's Theory (2)]
\label{bezoutiansignature}
Let $P, Q \in \R[y]$ with $\deg P = p \ge 1$, 
Then 
$$
{\rm Si}({\rm Her}(P;Q))  = {\rm PmV}({\rm hmi}
(P; Q)).
$$
\end{theorem}

\begin{proof}{Proof:}
See \cite[Theorem 4.33, Proposition 4.55 and Lemma 9.26]{BPRbook}
or \cite[Section 5.2]{zerelem} for a 
proof based on algebraic identities. 
\end{proof}

\subsection{Sign determination}

Consider now $P \in \R[y]$ and  $\mathcal{P} =P_{1} , \ldots ,P_{s}$, a finite list of polynomials in
$\R[y]$.
Let $\sigma$ be a sign condition on
$\mathcal{P}$. 
The cardinality of 
$$
\{\theta  \in \R \ | \ P(\theta) = 0, \, \sign(\cP(\theta)) = \sigma \}
$$
is denoted by $c(\sigma, \{\theta \in \R \ | \ P(\theta) = 0\})$
or simply by $c(\sigma)$ if the polynomial $P$ is fixed and clear from the context. 
Note that if
$$
\{\theta  \in \R \ | \ P(\theta) = 0, \, \sign(\cP(\theta)) = \sigma \}=\emptyset,
$$
then $c(\sigma)=0$.

The (univariate) \emph{Sign Determination problem} is to determine $c(\sigma)$ for every sign condition
$\sigma$ on $\cP$.  
It is a basic algorithmic problem for real numbers which has been studied extensively
(see for example \cite{Tarski51,BKR,Canny93b,BPRbook}).

There is a very close relation between the sign determination problem and Tarski queries.

\begin{proposition}\label{prop:enough_TQ}
Let $P \in \R[y]$ with $\deg P = p \ge 1$  and  $\mathcal{P} =P_{1} , \ldots ,P_{s}$  a finite list of polynomials in
$\R[y]$. 
The list of all Tarski-queries 
$$
\left[{\rm TaQu} (Q;P), Q \in  
\Big\{ \displaystyle{\prod_{1\le h \le s} P_h^{\alpha_h}}  
\ | \ 
(\alpha_1, \dots, \alpha_s)
\in {\{0,1,2\}} ^{\{1,\ldots,s\}}, \, \#\{h  \ | \ \alpha_h \ne 0\}\le {\rm bit}(p) \Big\}\right]
$$
determines  the cardinality $c(\sigma)$ for every sign condition $\sigma$ on $\cP$ at the roots of $P$.
\end{proposition}

\begin{proof}{Proof:}
The sign determination procedure described in 
\cite[Algorithm 10.11]{BPRbook} proceeds in $s$ steps as follows: for $i = 1, \dots, s$, 
at step $i$, the cardinality $c(\sigma)$ for every 
sign condition $\sigma$ on $P_1, \dots, P_i$ is computed. 
In order to do so, at each step, first several Tarski queries 
${\rm TaQu} (Q;P)$ are calculated, and then an invertible linear system with entries in $\mathbb{Z}$ is solved. 
By \cite[Proposition 10.74]{BPRbook},
every 
polynomial $Q$ such that 
${\rm TaQu} (Q;P)$ 
is  
calculated along the execution of the algorithm, 
is a product of at most ${\rm bit}(p)$
of the polynomials $P_1, \dots, P_s$ each of them raised to the power $1$ or $2$. 
Therefore, 
once all
 Tarski-queries ${\rm TaQu} (Q;P)$ with 
$$
Q \in  
\Big\{ \displaystyle{\prod_{1\le h \le s} P_h^{\alpha_h}}  
\ | \ 
 (\alpha_1, \dots, \alpha_s)
\in {\{0,1,2\}} ^{\{1,\ldots,s\}}, \, \#\{h  \ | \ \alpha_h \ne 0\}\le {\rm bit}(p) \Big\}
$$
are 
known,
 the output of the algorithm, which is
the cardinality $c(\sigma)$ for every  sign condition $\sigma$ on $P_1, \dots, P_s$ 
at the zeroes of $P$, is  determined.
\end{proof}

As it was said before, in this paper
we do not consider the cost of doing arithmetic operations between 
auxiliar numerical quantities (such as cardinalities of sets).
Nevertheless, 
we refer to \cite[Section 10.3]{BPR3} for details on specific methods to solve the integer linear systems 
involved in the sign determination algorithm cited in the proof of Proposition 
\ref{prop:enough_TQ}, as well as bounds on its bit complexity.

\begin{remark}
\label{cor:thom_enc_sign}
Given $P, Q \in \R[y]$ with $\deg P = p \ge 1$,  
solving the sign determination problem for
the list $\Der(P')$ (see Definition \ref{def:Thom}) 
means to compute the Thom encodings of the real roots of $P$.
Solving the sign determination problem for
the list $[\Der(P') , Q]$ (the list
$\Der(P')$ extended with the polynomial $Q$) means to additionally compute the sign of $Q$ at each of the real 
roots of $P$, encoded by their Thom encoding. 
\end{remark}

In view of Proposition \ref{prop:enough_TQ} and Remark \ref{cor:thom_enc_sign} we consider the following 
Notation and Definition. 

\begin{notation} \label{notaDPTP0}
Let $\A$ be a commutative ring, $P, Q \in \A[y]$ with $\deg P = p \ge 1$  and
$j \in \N$. We define
$$
\begin{array}{rcl}
 {\rm PDer}_j(P) &=& \Big\{ \displaystyle{\prod_{1\le h \le p-1} (P^{(h)})^{\alpha_h}}  
\ | \ 
\alpha
\in {\{0,1,2\}} ^{\{1,\ldots,p-1\}}, \, \#\{h  \ | \ \alpha_h \ne 0\}\le j \Big\}
\subset \A[y], \\[4mm]
{\rm PDer}_{j}(P;Q) &= &\{AB \ | \ A \in {\rm PDer}_{j}(P), \, 
B \in \{Q, Q^2\} \} \subset \A[y].
\end{array}
$$
\end{notation}

\begin{defn} \label{notaDPTP}
Let $P, Q \in \R[y]$ with $\deg P = p \ge 1$. We define
$$
\begin{array}{rcl}
{\rm thelim}(P)& = &\displaystyle{\bigcup_{A \in 
{\rm PDer}_{{\rm bit}(p)}(P)} } {\rm hmi}(P; A)\subset \R,
\\[4mm]
{\rm thelim}(P;Q)& = & \displaystyle{\bigcup_{A \in {
{\rm PDer}}_{{\rm bit}(p)-1}(P;Q)} } {\rm hmi}(P; A)\subset \R.
\end{array}
$$
\end{defn}

\begin{corollary}\label{thm:prelim_signdet} Let $P, Q \in \R[y]$ with $\deg P = p \ge 1$.
The list of signs of
${\rm thelim}(P)$ and ${\rm thelim}(P; Q)$
determines the Thom encoding of the 
real roots of $P$ and the sign of $Q$ at each of these roots.
\end{corollary}
\begin{proof}{Proof:}
Consider $\cP = P_1, \dots, P_p = [\Der(P'), Q]$;
we have that 
$${\rm PDer}_{{\rm bit}(p)}(P) \ \cup \
{\rm PDer}_{{\rm bit}(p)-1}(P;Q)
=
$$
$$
= \Big\{ \displaystyle{\prod_{1\le h \le p} P_h^{\alpha_h}}  
\ | \ 
\alpha
\in {\{0,1,2\}} ^{\{1,\ldots,s\}}, \, \#\{h  \ | \ \alpha_h \ne 0\}\le {\rm bit}(p) \Big\}.
$$
The result follows then from Theorem \ref{hermite}, Theorem \ref{bezoutiansignature} and 
Proposition \ref{prop:enough_TQ}.
\end{proof}

Note that the results we present here are not optimal in the number of Tarski queries to be considered, but
they are instead well adapted to the parametric case. For a more refined sign determination process see 
\cite[Chapter 10]{BPRbook}.

\section{Eliminating one variable}\label{sec:projecting_one}

In this section, we consider a set of variables $u = (u_1, \dots, u_{\ell})$ which we take as parameters, 
and  a single variable $y$ which we take as the main variable.
In order to study the  elimination of the variable $y$, 
we first review sign determination in a parametric context.

Through this section, 
derivative, degree and leading coefficient are taken with respect to $y$. 
For $P \in \D[u, y]$ we  denote by $\deg P$ and $\deg_u P$ 
its degree with respect to $y$ and to $u$ respectively. For a finite family
$\cF \subset \D[u, y]$, we denote by $\deg \cF$ and $\deg_u \cF$  the maximum
of $\deg P$ and $\deg_u P$ for $P \in \cF$ respectively.

\subsection{Parametric Thom encoding and sign determination}

Given $P, Q \in \D[u, y]$, we want to describe polynomial conditions on the parameters 
fixing the Thom encoding of the real roots of $P$ and the sign of $Q$ at 
each of them. 
The first problem to consider in this parametric context is that some specializations of the
parameters may cause a drop in the degree of $P$, which is particularly important since this degree
fixes the size of the Hermite's matrix of $P$ and $Q$. Note that, on the other hand, 
specializations of the
parameters causing a drop in the degree of $Q$ do not cause any problem. 

\begin{defn}\label{1:def:truncation}
Let $P(u, y)=\sum_{0 \le h \le p}c_h(u)y^h \in \D[u, y]$ with $p \ge 0$ and $c_p(u) \not \equiv 0$. 

For $ -1\le j \le p$,  the 
\emph{truncation of} ${P}$ \emph{at} ${j}$
is
$$
	{\rm Tru}_j(P)=c_j(u)y^j+\ldots+c_0(u) \in \D[u, y].
$$

The \emph{set of truncations
 of} $P$ is
the finite subset of $\D[u, y]$
defined inductively on the degree of $P$ by
${\rm Tru}(0)=\emptyset$ and 
$$
 {\rm Tru}(P)=  
\begin{cases}
    \{P\} & \mbox{if }  {\rm lc}(P) \in \D 
       \\
   \{P\}\cup {\rm Tru}({\rm Tru}_{p-1}(P)) & \mbox{otherwise}.
     \end{cases} 
$$

The \emph{set of relevant coefficients of} $P$ is the finite subset of $\D[u]$ defined 
inductively on the degree of $P$ by 
${\rm RC}(0)=\emptyset$ and 
$$
 {\rm RC}(P)=  
\begin{cases}
    \emptyset & \mbox{if }  {\rm lc}(P) \in \D,\\
   \{ {\rm lc}(P)\}\cup {\rm RC}({\rm Tru}_{p-1}(P)) & \mbox{otherwise}.
     \end{cases} 
$$
\end{defn}

The idea behind Definition \ref{1:def:truncation} is that the degree of $P$ is fixed 
once the sign of the relevant coefficients of $P$ is known. 

Another problem arising in the parametric context is 
that we want to eliminate the variable $y$ keeping conditions on the parameters $u$ 
defined by polynomials rather than rational functions.
Therefore, we consider the following definition.

\begin{notation}
Let $P(u,y) = \sum_{0 \le h \le p}c_h(u)y^h, Q= \sum_{0 \le h \le q}c'_h(u)y^h \in \D[u, y]$ 
with $p \ge 1$ and $c_p(u) \not \equiv 0$. 
As in Definition \ref{def:her_mat}   
we consider the matrix 
${\rm Her}(P;Q) \in \D(u)^{p \times p}$.
Taking into account Remark \ref{rem:degree_her_mat} and following Notation \ref{not:Hermite_minors}, for $0 \le j \le p-1$, we denote by
$$
{\rm HMi}_j(P;Q) = c_p(u)^{(p-j)(q+2p-2)}{\rm hmi}_j(P;Q) \in \D[u].
$$
We denote by   
${{\rm HMi}}(P;Q)$ the list 
$$
[{{\rm HMi}}_{0}(P;Q), \dots, {{\rm HMi}}_{p-1}(P;Q)]
\subset \D[u].
$$ 
\end{notation}

\begin{lemma} \label{lem:complexity_berk}
$$
\deg_u {{\rm HMi}}(P;Q) \le p \Big( (q+2p-2) \deg_u P  +\deg_u Q \Big).
$$
Moreover, given the matrix
$c_{p}(u)^{q + 2p -2}{\rm Her}(P; Q)$, the computation of ${{\rm HMi}}(P;Q)$
can be done  
in
$O(p^4)$
operations in $\D[u]$, each of them  between polynomials of degree bounded by 
$p ( (q+2p-2) \deg_u P  +\deg_u Q )$.
\end{lemma}

\begin{proof}{Proof:}
The degree bound for ${{\rm HMi}}(P;Q)$ follows from the fact that 
${{\rm HMi}}(P;Q)$ is the list of principal minors of the matrix
$c_{p}(u)^{q + 2p -2}{\rm Her}(P; Q) \in \D[u]^{p \times p}$ and the 
degree bound from Remark
\ref{rem:degree_her_mat}.

For the bound on the number of operations in $\D[u]$
and the degree bound in intermediate computations, we simply use 
Berkowitz Algorithm (see \cite{Berk}), taking into account that along the execution of this division free algorithm 
for the computation of the determinant of a given matrix,
all its principal minors are recursively computed. 
\end{proof}

Now we consider the following definitions.

\begin{defn} \label{not:DPTP_parametric}
Let $P, Q \in \D[u, y]$ with $\deg P = p$. If $p \ge 1$, following Notation \ref{notaDPTP0}, we define
$$
\begin{array}{rcl}
{{\rm ThElim}}(P)& = &\displaystyle{\bigcup_{A \in 
{\rm PDer}_{{\rm bit}(p)}(P)} } {{\rm HMi}}(P; A)\subset \D[u],
\\[4mm]
{{\rm ThElim}}(P;Q)& = & \displaystyle{\bigcup_{A \in {
{\rm PDer}}_{{\rm bit}(p)-1}(P;Q)} } {{\rm HMi}}(P; A)\subset \D[u].
\end{array}
$$
If $p = 0$ (i.e., $P \in \D[u]$), we define  
${{\rm ThElim}}(P)$ and ${{\rm ThElim}}(P; Q)$ as the empty lists.

Finally, we define
$$
\begin{array}{rcl}
{\rm Elim}(P;Q)& = & {\rm RC}(P) \cup \displaystyle{ \bigcup_{T \in {\rm Tru}(P)}  \Big(    {{\rm ThElim}}(T) 
\cup  {{\rm ThElim}}(T;Q)} \Big).
\end{array}
$$
\end{defn}

We can prove now the following result. 

\begin{proposition} \label{thm:fixing_thom_and_sign}
Let $P, Q \in \D[u, y]$ with $P \not \equiv 0$. For every $\upsilon \in \R^{\ell}$, 
the 
realizable
sign condition 
on the family
$$ 
{\rm Elim}(P;Q)
\subset \D[u]
$$
satisfied by $\upsilon$
determines 
the fact that $P(\upsilon, y) \equiv 0$ or 
$P(\upsilon, y) \not\equiv 0$, and,
if  $P(\upsilon, y) \not\equiv 0$,
it also determines
the 
Thom encoding of the 
real roots of $P(\upsilon, y)$ and the sign of $Q(\upsilon, y)$ at each of these roots.
\end{proposition}

\begin{proof}{Proof:} 
Let $p = \deg P$. It is clear that 
the fact that $P(\upsilon, y) \equiv 0$ or $P(\upsilon, y) \not\equiv 0$
is determined by the sign condition on 
${\rm RC}(P)$ satisfied by $\upsilon$. From now we suppose that 
$P(\upsilon, y) \not \equiv 0$. Again, it is also clear that 
the degree of $P(\upsilon, y) \le p$
is also determined by the sign condition on 
${\rm RC}(P)$ satisfied by $\upsilon$; we call $p'$ this degree. 
If $p' = 0$ then $P(\upsilon, y)$ has no real root; so from now we suppose $p' \ge 1$ 
and we only keep the 
information given by the sign
condition satisfied by $\upsilon$ on
$$
{{\rm ThElim}}(T) 
\cup {{\rm ThElim}}(T;Q)
$$
for 
$$
T = {\rm Tru}_{p'}(P)
\in 
{\rm Tru}(P).
$$

Now, for $0 \le j \le p'-1$ and  $A \in {\rm PDer}_{{\rm bit}(p')}(T)$ or 
$A \in {\rm PDer}_{{\rm bit}(p')-1}(T;Q)$
we have that 
$$
{{\rm HMi}}_j(T; A)
= 
c_{p'}(u)^{(p'-j)(q+2p'-2)}{\rm hmi}_j(T;A).
$$
It is the case that either $c_{p'}(u) \in \D$ or 
$ c_{p'}(u)\in {\rm RC}(P)$, but in any situation
the sign of $c_{p'}(\upsilon)$ is known, and then the sign of every element in 
${\rm hmi}_j(T(\upsilon, y);A(\upsilon, y))$ is also known.  

Finally, by Corollary \ref{thm:prelim_signdet}
this is enough to determine 
the Thom encoding of the 
real roots of $T(\upsilon, y)$ and the sign of $Q(\upsilon, y)$ at each of these roots.
\end{proof}

\begin{remark}\label{rem:prop_based_on_alg}
The proof of correctness of Proposition \ref{thm:fixing_thom_and_sign} is based on the determination of Thom 
encoding of real roots and the sign of another polynomial at these roots; 
thus, this proof is entirely based on algebra. For instance, there is no need of 
sample points meeting every connected component of the realization of sign conditions. 
\end{remark}

\begin{remark}\label{rem:degree_and_number_thelim} 
Let $P, Q \in \D[u, y]$ with $\deg P = p \ge 1$ and $\deg Q = q$. 
Following Notation \ref{notaDPTP0},  there are
$$
\sum_{0 \le h \le j} {{p-1}\choose{h}}2^h \le 2p^{j}
$$
elements in ${\rm PDer}_j(P)$.
Therefore, there are 
at most $2p^{{\rm bit}(p) + 1}$ elements in 
${{\rm ThElim}}(P)$ and by Lemma \ref{lem:complexity_berk}
their degree in $u = (u_1, \dots, u_\ell)$ are bounded by 
$$
p\Big(
(2(p-1){\rm bit}(p) + 2p - 2)\deg_u P + 2{\rm bit}(p)\deg_u P
\Big) \le 2p^2( {\rm bit}(p)+1)\deg_u P.
$$
Similarly, there are at most
$4p^{{\rm bit}(p)}$  elements in ${{\rm ThElim}}(P;Q)$ and 
their degree in $u = (u_1, \dots, u_\ell)$ are bounded by 
\begin{eqnarray*}
&&p\Big(
(2(p-1)({\rm bit}(p)-1) + 2q + 2p - 2)\deg_u P
+
2({\rm bit}(p)- 1)\deg_u P + 2 \deg_u Q 
\Big) =
\\
&=&
p\Big(
( 2p{\rm bit}(p) + 2q - 2)  \deg_u P + 2\deg_u Q
\Big).
\end{eqnarray*}
\end{remark}

\subsection{Fixing the realizable sign conditions on a family}

In order to fix the realizable sign conditions on a parametric family of univariate 
polynomials, we consider the following 
definition.

\begin{defn}\label{def:sign_det_sect_three_first} 
Let $\cF$ be
a finite family of polynomials in $\D[u,y]$. 
We denote by 
$$
 \Der({\cal F}) = \bigcup_{P \in {\cal F} \setminus \D[u]}\Der(P) \subset \D[u,y].
$$
We define 
$$ 
{\rm Elim}({\mathcal F}) =  
\bigcup_{ P\in {\cal F} \setminus \{0\}} 
\left(
{\rm RC}(P) \cup \bigcup_{T \in {\rm Tru}(P)}  \displaystyle{\Big(    {{\rm ThElim}}(T) 
\cup  \bigcup_{Q \in  \Der(\cF  \setminus \{P\})} {{\rm ThElim}}(T;Q)} \Big) \right)
\subset \D[u].
$$
\end{defn}

We prove now the following result.

\begin{proposition}\label{thm:elimination_one_variable}
Let $\cF$ be
a finite family of polynomials in $\D[u,y]$. 
For every $\upsilon \in \R^{\ell}$, 
the 
realizable
sign condition 
on the family 
$$ 
{\rm Elim}({\mathcal F}) \subset \D[u]$$
satisfied by $\upsilon$
determines the list 
${\rm SIGN}(\cF(\upsilon, y))$.
\end{proposition}

\begin{proof}{Proof:} 
By Proposition \ref{thm:fixing_thom_and_sign}, the sign condition on ${\rm Elim}({\mathcal F})$
satisfied by $\upsilon$ determines for every $P \in 
\cF \setminus \{0\}$ the fact that $P(\upsilon, y) \equiv 0$ or $P(\upsilon, y) \not\equiv 0$, and,
if  $P(\upsilon, y) \not\equiv 0$, 
it also determines the Thom encoding of the real roots of $P(\upsilon, y)$ 
and, for every $Q \in \Der(\cF  \setminus \{P\})$, the sign of $Q(\upsilon, y)$ at each of these
real roots. 

Now, for each $P \in \cF$ with $P(\upsilon, y) \not \equiv 0$, since the Thom encoding of 
the real roots of $P(\upsilon, y)$ is known and 
the sign of the leading coefficient of $P(\upsilon, y)$ 
is also known, 
we can deduce the multiplicity of each real root
and, 
by Proposition \ref{lemma_order_thom_encoding}, also the order between them.
All this information is enough to 
determine the sign of $P(\upsilon, y)$ on every (bounded or unbounded) interval of the real 
line defined by the 
real roots of $P(\upsilon, y)$. 

Finally, in order to determine the list ${\rm SIGN}(\cF(\upsilon, y))$ we only need to know how to 
order
the real roots coming from different polynomials
$P_1(\upsilon, y)$ and $P_2(\upsilon, y)$  in $\cF(\upsilon, y)$. 
Once again by Proposition \ref{thm:fixing_thom_and_sign}, the signs of $\Der(P_1(\upsilon, y))$ 
at all the real roots of 
$P_1(\upsilon, y)$ and $P_2(\upsilon, y)$ is known. 
The only detail to take into account is that if it happens that 
$\deg P_1(u, y) = \deg P_1(\upsilon, y) = p_1$ 
we also need to know the sign of $P_1(\upsilon, y)^{(p_1)}$ 
at all the real roots of 
$P_1(\upsilon, y)$ and $P_2(\upsilon, y)$ to be able to order them
(and by definition, $P_1(u, y)^{(p_1)} \not \in \Der(P_1(u, y))$). 
Nevertheless, this 
is indeed the case since the leading coefficient $c_{p_1}(u)$ of $P_1(u, y)$
is either in $\D$ or in ${\rm RC}(P_1)$ and in any situation the sign 
of $c_{p_1}(\upsilon)$ is known. 
Finally we can order the real roots of
$P_1(\upsilon, y)$ and $P_2(\upsilon, y)$ using once again Proposition \ref{lemma_order_thom_encoding}.
\end{proof}

\begin{remark}\label{elimalgebraic}
As in Proposition \ref{thm:fixing_thom_and_sign} (see Remark \ref{rem:prop_based_on_alg}), 
the proof of correctness of Proposition  \ref{thm:elimination_one_variable} is entirely based on algebra. 
No geometric concept is needed. 
\end{remark}

\begin{lemma}\label{lem:degree_and_number_elim} Let $\cF$ be
a family of $s$ polynomials in $\D[u,y]$ with $\deg \cF = p$. 
If $p = 0$ (i.e., $\cF \subset \D[u]$) then ${\rm Elim}(\cF) = \cF \setminus \D.$ If $p \ge 1$, 
there are at most 
$$
4s^2p^{{\rm bit}(p) + 2}
$$
elements in ${\rm Elim}({\mathcal F})$, their degree in $u = (u_1, \dots, u_\ell)$ are bounded by 
$$
4p^3\deg_u \cF 
$$
and
the complexity of 
computing ${\rm Elim}(\cF)$ is
$$O(s^2p^{{\rm bit}(p)+5})$$
operations in $\D[u]$, each of them between polynomials of degree at most $4p^3\deg_u \cF$.
\end{lemma}

\begin{proof}{Proof:} If $p = 0$ there is nothing to prove, so we suppose $p \ge 1$. 
By Remark \ref{rem:degree_and_number_thelim}, 
there are at most 
$$
s\left(
p+1 + p \left(
2p^{{\rm bit}(p) + 1} + (s-1)4p^{{\rm bit}(p)+1}
\right)
\right) 
\le
4s^2p^{{\rm bit}(p) + 2}
$$
elements in ${\rm Elim}({\mathcal F})$ and their degree in $u = (u_1, \dots, u_\ell)$ are bounded by 
$$
2p^2( {\rm bit}(p)+1)\deg_u \cF \le 
4p^3\deg_u \cF.
$$

The computation of ${\rm RC}(P)$ for every $P \in \cF \setminus\{0\}$ is cost free. 
There are at most $sp$ polynomials $T \in {\rm Tru}(P)$ for some 
$P\in \cF \setminus\{0\}$ to consider, and for each of these polynomials $T$ we have to compute
$$
{{\rm ThElim}}(T) 
\cup  \bigcup_{Q \in  \Der(\cF  \setminus \{P\})} {{\rm ThElim}}(T;Q)
$$
(this is so since if $\deg T = 0$ then ${{\rm ThElim}}(T)$
and ${{\rm ThElim}}(T;Q)$ are the empty lists). 

From now, we consider a fixed $T(u,y) = \sum_{0 \le h \le p'} c_h(u)y^h \in \D[u, y]$, 
with $1 \le p' \le p$ and $c_{p'}(u) \not \equiv 0$.  
We consider also the basis ${\cal B} = \{1, y, \dots, y^{p'-1}\}$ of the $\R(u)$-vector space
$V = \R(u)[y]/T(u,y)$. For $h \in \N$, we 
define $M_h \in \D[u]^{p' \times p'}$ as the matrix in basis ${\cal B}$ of the linear mapping
of multiplication by $(c_{p'}(u)y)^h$ in $V$. 
It is clear that
$$
M_1 = \left(\begin{array}{ccccc}
       0 & \cdots & \cdots & 0 & -c_0(u) \cr
       c_{p'}(u) & \ddots &  & \vdots & -c_1(u) \cr
      0 & c_{p'}(u) & \ddots & \vdots & \vdots \cr
\vdots & \ddots & \ddots & 0 & \vdots \cr
0 &\cdots & 0  & c_{p'}(u) & -c_{{p'}-1}(u)
      \end{array}
 \right)
$$
and also that $\deg_u M_h \le h \deg_u P$.

We need to compute ${{\rm HMi}}(T; A)$ for a number of polynomials $A \in \D[u, y]$
with $\deg A \le 2p{\rm bit}(p)$.  
If $A = \sum_{0 \le h \le a} c'_h(u)y^h$, 
then for $0 \le j \le p'-1$, 
${{\rm HMi}}_j(T; A)$ is the $(p'-j)$-th principal minor of the matrix
$$c_{p'}(u)^{a + 2p' -2}{\rm Her}(T; A),$$
and for $1 \le j_1, j_2 \le p'$ we have that 
\begin{equation}\label{eq:rec_formula}
c_{p'}(u)^{a + 2p' -2}{\rm Her}(T; A)_{j_1, j_2} = \sum_{0 \le h \le a}
c'_h(u) c_{p'}(u)^{a + 2p' - h - j_1 - j_2 }{\rm Tra}(M_{h+j_1+j_2 - 2}).
\end{equation}

We describe first the part of the computation which depends on $T$ but not on $A$,
then we describe the computation of all the polynomials $A$ we need to consider, and finally we describe the part 
of the computation which depends on both $T$ and $A$. 

{\it First Step:} 
Our aim is to compute ${\rm Tra}(M_h)$ for $0 \le h \le 2p{\rm bit}(p) + 2p - 2$.

The matrices $M_0$ and $M_1$ are already known and computing their traces
is 
cost free. We compute then
${\rm Tra}(M_h)$ for $2 \le h \le p'-1$ 
and then 
we proceed using a recursive formula to compute all the remaining required traces.

Successively computing
$$M_2=M_1\cdot M_1, \ \ldots \ , M_{h+1}=M_1 \cdot M_h, \ \ldots \, M_{p'-1}=M_1\cdot M_{p'-2},$$
and 
taking into account that $M_1$ has at most $2$ non-zero entries per row, we can compute $M_h$ 
for $2 \le h \le p'-1$
within $O(p^3)$
operations in $\D[u]$, and the traces of all these matrices within 
$ O(p^2) $
operations in $\D[u]$.

For $h \ge p'$ we have that
$$
{\rm Tra}(M_h) = - \sum_{1 \le i \le p'}c_{p'-i}(u)c_{p'}(u)^{i-1}{\rm Tra}(M_{h - i}).
$$
We compute $c_{p'}(u)^{i-1}$ for $1 \le i \le p'$ within $O(p)$ operations in $\D[u]$, then we successively compute 
${\rm Tra}(M_h)$ for $p' \le h \le 2p{\rm bit}(p) + 2p - 2$
within $O(p^2{\rm bit}(p))$ operations in $\D[u]$. 

Finally, the whole step can be done within $O(p^3)$
operations in $\D[u]$. Since for $h\in \N$ we have that $\deg_u M_h  \le h \deg_u P$, these operations
are between polynomials of degree at most $(2p{\rm bit}(p) + 2p - 2) \deg_u \cF \le 4p^3\deg_u \cF$. 

{\it Second Step:} 
Now we proceed to the computation of all the polynomials $A$. 

Following Remark \ref{rem:degree_and_number_thelim} there are at most 
$2p^{{\rm bit}(p)}$ polynomials $A \in {\rm PDer}_{{\rm bit}(p')}(T)$. We  compute
all of them starting from the constant polynomial $1$ and then multiplying each time 
a derivative of $T$, of degree at most $p$,
and
a previously computed polynomial in  ${\rm PDer}_{{\rm bit}(p')}(T)$, of degree at most $(2{\rm bit}(p)-1)p$.
In this way, we compute all the polynomials
in ${\rm PDer}_{{\rm bit}(p')}(T)$ within
$O(p^{{\rm bit}(p) + 3} ) $
operations in $\D[u]$. 

Similarly, for each polynomial $Q \in {\rm Der}(\cF \setminus \{P\})$
there are at most $4p^{{\rm bit}(p)-1}$ polynomials $A \in {\rm PDer}_{{\rm bit}(p')-1}(T; Q)$ 
and we  compute all of them multiplying each time $Q$ by a previously computed polynomial in 
${\rm PDer}_{{\rm bit}(p')-1}(T) \cup {\rm PDer}_{{\rm bit}(p')-1}(T, Q)$, of degree at most $(2{\rm bit}(p)-1)p$.
In this way, we compute all the polynomials
in ${\rm PDer}_{{\rm bit}(p')-1}(T; Q)$ within
$O(p^{{\rm bit}(p) + 2} ) $
operations in $\D[u]$.

Finally, since there are at most $(s-1)p$ polynomials $Q \in {\rm Der}(\cF \setminus \{P\})$, 
the
whole step can be done within 
$O(sp^{{\rm bit}(p) + 3})$
operations in $\D[u]$. 
It is clear that all these operations are between polynomials of degree at most $2p{\rm bit}(p) \deg_u \cF \le 4p^3\deg_u \cF$.

{\it Third Step:} We have to compute 
${{\rm HMi}}(T; A)$
for every $A \in {\rm PDer}_{{\rm bit}(p')}(T)$
and also, for every $Q \in {\rm Der}(\cF \setminus \{P\})$ and every $A 
\in {\rm PDer}_{{\rm bit}(p')-1}(T; Q)$. 
By Remark \ref{rem:degree_and_number_thelim}, there are then $O(sp^{{\rm bit}(p)})$ polynomials $A$ to consider.

For a fixed $A = \sum_{0 \le h \le a} c'_h(u)y^h$ with $a \le 2p{\rm bit}(p)$,
in order to compute he matrix
$c_{p'}(u)^{a + 2p' -2}{\rm Her}(T; A)$, 
we first compute $c_{p'}(u)^{i}$ for $1 \le i \le a + 2p' - 2$ within $O(p{\rm bit}(p))$ operations in $\D[u]$.
Then, using equation (\ref{eq:rec_formula}), we can compute each entry of this matrix within $O(p)$ operations since
all the required traces have already been computed. Note that since 
$c_{p'}(u)^{a + 2p' -2}{\rm Her}(T; A)$ is a Hankel matrix, we only need to compute $2p' - 
1$ entries. 
Therefore, the computation of $c_{p'}(u)^{a + 2p' -2}{\rm Her}(T; A)$ can be done within $O(p^2)$
operations in $\D[u]$, each of them between polynomials of degree at most $(2p{\rm bit}(p) + 2p -1)\deg_u \cF$.

The last part of the step is to compute the principal minors of 
$c_{p'}(u)^{a + 2p' -2}{\rm Her}(T; A)$.
By Lemma \ref{lem:complexity_berk} this can be done 
within 
$
O(p^4)
$
operations in $\D[u]$, each of them between polynomials of degree at most 
$p(2p{\rm bit}(p) + 2p -1)\deg_u \cF$.

Finally, 
the whole step can be done within 
$
O(sp^{{\rm bit}(p) + 4})
$
operations in $\D[u]$, each of them between polynomials of degree at most $p(2p{\rm bit}(p) + 2p -1)\deg_u \cF \le 4p^3\deg_u \cF$.
\end{proof}

\section{Main result}\label{sec:main}

Let $\cF$  be a finite family in $ \D[x_1, \dots, x_k]$. 
In this section we 
define the families ${\rm Elim}_i(\cF)$  for $0 \le i \le k-1$ and we 
prove Theorem \ref{quantifier_elimination}.

The main idea is to repeatedly use the construction of ${\rm Elim}$ as in Definition 
\ref{def:sign_det_sect_three_first}, 
where for each $i = k-1, \dots, 0$ (taken in this order), the vector $u = (x_1, \dots, x_{i})$ will 
play
 the role of the set of parameters
and  $y = x_{i+1}$  will play the role of the main variable.

\begin{defn}\label{def:higher_order_eliminating_families}
We define ${\rm Elim}_{k}({\cal F})$ as ${\cal F}$.
Then, for $i = k-1, \dots, 0$, we define
inductively 
$$
{\rm Elim}_{i}({\cal F}) = {\rm Elim}({\rm Elim}_{i+1}({\cal F})) \subset \D[x_1, \dots, x_{i}].
$$
\end{defn}

\begin{proof}{Proof of Theorem \ref{quantifier_elimination}:}
The proof is based on a cylindrical structure on the realizable sign conditions 
described by the families ${\rm Elim}_i(\cF)$ for $1 \le i \le k$. 

First, we prove the existence of the quantifier free formula $\Psi$ in a constructive way.
To do so, we proceed in three steps.

The first step is to successively compute ${\rm Elim}_{k-1}(\cF), \dots,  
{\rm Elim}_{0}(\cF)$.

The second step is to successively compute $\SIGN({\rm Elim}_1(\cF)), \dots, 
\SIGN({\rm Elim}_{k}(\cF))$ together  
with some additional information that will be needed in the third step. More precisely,  
starting from  ${\rm Elim}_0(\cF)\subset \D$,
for $i=0, \dots, k-1$, we consider every $\tau \in \SIGN({\rm Elim}_i(\cF))$.   
Following the procedure described in the proof of Proposition \ref{thm:elimination_one_variable}, 
for each such $\tau$
we compute
$$
\SIGN({\rm Elim}_{i+1}(\cF)(\upsilon, x_{i+1}))
$$ 
for any 
$\upsilon \in \R^{i}$ such that 
$$\sign({\rm Elim}_i(\cF)(\upsilon)) = \tau,$$
and we keep the record that 
$\SIGN({\rm Elim}_{i+1}(\cF)(\upsilon, x_{i+1})) \subset
\SIGN({\rm Elim}_{i+1}(\cF))$ is exactly the set of realizable sign conditions 
on the family ${\rm Elim}_{i+1}(\cF)$ given the extra condition that 
$\sign({\rm Elim}_{i}(\cF)) = \tau$.

Note that these first two steps only depend on $\cF$ rather than depending on
the given first order formula
$$
{\rm Qu}_{i+1} x_{i+1} \dots {\rm Qu}_{k} x_{k} \, \Phi(x_1, \dots, x_k).
$$

The last step is to compute $\Psi$, or what is equivalent,  $T_{\Phi}$. To do so, 
we proceed by reverse induction on $i = k-1, \dots, 1$. 

For $i = k-1$, we are given 
a first order formula of type 
$$
{\rm Qu}_{k} x_{k} \, \Phi(x_1, \dots, x_k).
$$
By Proposition \ref{thm:elimination_one_variable},
for every $\upsilon = (\upsilon_1, \dots, \upsilon_{k-1}) \in \R^{k-1}$, 
the list 
${\rm SIGN}(\cF(
\upsilon, x_k))$
is determined by 
the  
realizable
sign condition 
on the family 
$$ 
{\rm Elim}_{k-1}({\mathcal F}) \subset \D[x_1, \dots, x_{k-1}]$$
satisfied by $\upsilon$. Since $\Phi(x_1, \dots,  x_k)$
is a quantifier free formula with atoms defined by polynomials in $\cF$, 
from $\SIGN({\cal F}(
\upsilon, x_k))$ it is possible to decide the truth value of the formula
$$
{\rm Qu}_k x_k \, \Phi(
\upsilon, x_k).
$$
So, we define 
$$
T_\Phi=\{\tau \in
\SIGN({\rm Elim}_{k-1}(\cF))
  \mid \forall 
\upsilon \in {\rm Real}(\tau, \R) , {\rm Qu}_k x_k \Phi(
\upsilon, x_k) \hbox{{
is true}}\}
$$
 and
$$\begin{array}{rcl}
\Psi(x_1, \dots, x_{k-1})& =& \displaystyle{\bigvee_
{\tau \in T_\Phi} \left(
\sign({\rm Elim}_{k-1}(\cF)) = \tau \right) }
\end{array}$$
and we are done.

Now, we take $1 \le i \le k-2$ and we are given 
a first order formula of type 
$$
{\rm Qu}_{i+1} x_{i+1} \, {\rm Qu}_{i+2} x_{i+2}  \dots   {\rm Qu}_{k} x_{k} \, \Phi(x_1, \dots, x_k).
$$
By the inductive hypothesis, 
there exists a quantifier free formula 
$$
\Psi'(x_1, \dots, x_{i+1})
$$ 
with atoms in 
${\rm Elim}_{i+1}(\cF)$ 
which is equivalent to 
$$
{\rm Qu}_{i+2} x_{i+2} \dots {\rm Qu}_{k} x_{k} \, \Phi(x_1, \dots, x_k).
$$
By Proposition \ref{thm:elimination_one_variable},
for every $\upsilon = (\upsilon_1, \dots, \upsilon_{i}) \in \R^{i}$, 
the list 
${\rm SIGN}({\rm Elim}_{i+1}(\cF)(
\upsilon, 
x_{i+1}
))$
is determined by 
the sign condition 
on the family 
$$ 
{\rm Elim}_{i}({\mathcal F}) \subset \D[x_1, \dots, x_{i}]$$
satisfied by $\upsilon$.
Since $\Psi'(x_1, \dots,  x_{i+1})$
is a quantifier free formula with atoms defined by polynomials in ${\rm Elim}_{i+1}(\cF)$, 
from $\SIGN({\rm Elim}_{i+1}(\cF)(
\upsilon,
 x_{i+1}
))$ 
it is possible to decide
the truth value of the formula
$$
{\rm Qu}_{i+1}x_{i+1
}  \, \Psi'(
\upsilon, x_{i+1}).
$$
Finally, we define  
$$T_{\Phi}=\{\tau \in 
\SIGN({\rm Elim}_i(\cF))
 \mid \forall 
\upsilon \in {\rm Real}(\tau, \R) , {\rm Qu}_{i+1
} x_{i+1
} \Psi'(
\upsilon,x_{i+1}) \hbox{{
 is true}}\}
$$ 
and
$$\begin{array}{rcl}
\Psi(x_1, \dots, x_{i})& = &
\displaystyle{\bigvee_{\tau \in T_{\Phi}}
\left( \sign({\rm Elim}_{i}(\cF)) = \tau \right) }
\end{array}$$
and we are done.

We now consider the quantitative part of the theorem.
First, using Lemma \ref{lem:degree_and_number_elim}, it can be easily proved by 
reverse induction that
for $i = k, \dots, 1$, for every $P \in  {\rm Elim}_{i}(\cF)$,  
$$
\deg P \le 4^{\frac{4^{k-i}-1}{3}}d^{4^{k-i}}.
$$

We  prove then, again using Lemma \ref{lem:degree_and_number_elim} 
and by reverse induction,
that for $i = k, \dots, 1$, 
$$
\# {\rm Elim}_{i}(\cF) \le s^{2^{k-i}}\max\{2, d\}^{(16^{k-i}-1){\rm bit}(d)}. 
$$
Indeed, $\# {\rm Elim}_{k}(\cF) = s$ and for $i = k-1, \dots, 1$, 
\begin{eqnarray*}
\# {\rm Elim}_{i}(\cF) &\le& 4 
s^{2^{k-i}} \max\{2, d\}^{2 (16^{k-i-1}-1){\rm bit}(d)}
(4^{\frac{4^{k-i-1}-1}{3}}d^{4^{k-i-1}})^{({\rm bit}(4^{\frac{4^{k-i-1}-1}{3}}d^{4^{k-i-1}}) + 2)} 
\le \\
&\le&
s^{2^{k-i}}\max\{2, d\}^{2 + 2(16^{k-i-1}-1){\rm bit}(d) + (2\frac{4^{k-i-1}-1}{3} + 4^{k-i-1})(2\frac{4^{k-i-1}-1}{3} + 4^{k-i-1}{\rm bit}(d) + 2)} 
\le \\
&\le& 
s^{2^{k-i}}\max\{2, d\}^{(16^{k-i}-1){\rm bit}(d)}. 
\end{eqnarray*}

Finally, we analyze the complexity of computing the quantifier free formula $\Psi$ following the procedure
we explained before.

Since the second and third step take only sign evaluation in $\D$ and operations in $\mathbb{Q}$, 
we only need to bound the complexity of the first step. 

One more time using Lemma \ref{lem:degree_and_number_elim}, for $1 \le i \le k-1$,  the computation of 
${\rm Elim}_{i}(\cF)$ from ${\rm Elim}_{i+1}(\cF)$  can be done within
$$
O\left(
s^{2^{k-i}} \max\{2, d\}^{2 (16^{k-i-1}-1){\rm bit}(d)}
(4^{\frac{4^{k-i-1}-1}{3}}d^{4^{k-i-1}})^{({\rm bit}(4^{\frac{4^{k-i-1}-1}{3}}d^{4^{k-i-1}}) + 5)} 
\right)
$$
operations in $\D[x_1, \dots, x_i]$ between polynomials of degree at most 
$$
4^{\frac{4^{k-i}-1}{3}}d^{4^{k-i}}.
$$
Taking into account that each of these operations can be done within 
$$
O\left(
4^{i2\frac{4^{k-i}-1}{3}}d^{i2\cdot 4^{k-i}}
\right)
$$
operations in $\D$ and 
$$\begin{array}{ll}
& s^{2^{k-i}} \max\{2, d\}^{2 (16^{k-i-1}-1){\rm bit}(d)}
(4^{\frac{4^{k-i-1}-1}{3}}d^{4^{k-i-1}})^{({\rm bit}(4^{\frac{4^{k-i-1}-1}{3}}d^{4^{k-i-1}}) + 5)} 
4^{i2\frac{4^{k-i}-1}{3}}d^{i2\cdot 4^{k-i}}
\le \\
\le &
s^{2^{k-i}}\max\{2, d\}^{2(16^{k-i-1}-1){\rm bit}(d) + (2\frac{4^{k-i-1}-1}{3} + 4^{k-i-1})(2\frac{4^{k-i-1}-1}{3} + 4^{k-i-1}{\rm bit}(d) + 5)
+ i4\frac{4^{k-i}-1}{3} + i2\cdot 4^{k-i}
} 
\le \\
\le &
s^{2^{k-i}}\max\{2, d\}^{{\rm bit}(d)(16^{k-i} + i 4^{k-i+1})}
,
\end{array}
$$
the computation of 
${\rm Elim}_{i}(\cF)$ from ${\rm Elim}_{i+1}(\cF)$  can be done within
$$
O\left(
s^{2^{k-i}}\max\{2, d\}^{{\rm bit}(d)(16^{k-i} + i 4^{k-i+1})}
\right)
$$
operations in $\D$.

On the other hand, similarly, the computation of 
${\rm Elim}_{0}(\cF)$ from ${\rm Elim}_{1}(\cF)$  can be done within
$$
O\left(
s^{2^{k}} \max\{2, d\}^{2 (16^{k-1}-1){\rm bit}(d)}
(4^{\frac{4^{k-1}-1}{3}}d^{4^{k-1}})^{({\rm bit}(4^{\frac{4^{k-1}-1}{3}}d^{4^{k-1}}) + 5)} 
\right) \le 
O\left(
s^{2^{k}} \max\{2, d\}^{16^{k}}
\right)
$$
operations in $\D$.

Finally, since 
$$
\displaystyle{\sum_{i = 1}^{k-1}s^{2^{k-i}}\max\{2, d\}^{{\rm bit}(d)(16^{k-i} + i 4^{k-i+1})}}
\le 
2\displaystyle{s^{2^{k}}\max\{2, d\}^{{\rm bit}(d)(16^{k} + (k-1) 4^{k+1})}},
$$
the complexity of the first step is 
$$
O \Big(
s^{2^{k}}\max\{2, d\}^{{\rm bit}(d)(16^{k} + (k-1) 4^{k+1})}
\Big)
$$
operations in $\D$.
\end{proof}

\begin{remark} \label{purelyalgebraic}
Note that the proof of correctness of the quantifier elimination method described in Theorem  is entirely based on Proposition \ref{thm:elimination_one_variable} and is thus completely algebraic.
\end{remark}

Note also that when the number of variables $k$ is fixed the complexity of our method is polynomial in the number $s$ of the polynomials, but is not polynomial  in the degree $d$ of the polynomials.
On the other hand, the complexity of the Cylindrical Algebraic Decomposition \cite{Col}  is polynomial in $s$ and $d$ when $k$ is fixed (see \cite[Chapter 11]{BPRbook}) .

\bigskip

\textbf{Acknowledgments:} We are very grateful to the referee for his/her relevant and useful suggestions and remarks.


\begin{thebibliography}{00}

\bibitem{BKR}
Ben-Or M., Kozen D. and  Reif J.
{\em The complexity of elementary algebra and geometry.}
J. of Computer and Systems Sciences 18, 251--264 (1986)

\bibitem{BPRQE}
Basu S., Pollack R. and  Roy M.-F. {\em On the
Combinatorial and Algebraic Complexity of Quantifier Elimination.}
 Journal of the A. C. M. 43, 1002--1045, (1996).

\bibitem{BPRbook} Basu S., Pollack R. and Roy M.-F.
{\em Algorithms in real algebraic geometry}, volume~10 of
{Algorithms and Computation in Mathematics}, Second edition.
Springer-Verlag, Berlin, 2006.


\bibitem{BPR3}  Basu S., Pollack R. and Roy M.-F.
{\em Algorithms in real algebraic geometry}, volume~10 of
{Algorithms and Computation in Mathematics}, New online version avialable at 
\texttt{ https://perso.univ-rennes1.fr/marie-francoise.roy/bpr-ed2-posted3.html}

\bibitem{Berk} Berkowitz S. 
{\em On computing the determinant in small parallel time using a small number of processors. } 
Inf. Process. Lett. 18, 147--150 (1984).

\bibitem{Canny93b} Canny J., 
{\em Improved Algorithms for Sign Determination and Existential Quantifier Elimination}, 
Comput. J. 36 (5), {409--418}, {1993}
          
\bibitem{coq} \texttt{https://coq.inria.fr/}

\bibitem{Coh} Cohen P. J. {\em Decision procedures for real and p-adic 
fields.} Comm. in Pure and Applied Math. 22,  131--151 (1969). 

\bibitem{Assia_Cyril}  Cohen C. and Mahboubi A. 
{\em Formal proofs in real algebraic geometry: from ordered fields to quantifier elimination.} 
Log. Methods Comput. Sci. 8 no. 1, 1:02 (2012).

\bibitem{Col} Collins G. {\em Quantifier Elimination for real closed 
fields by cylindric algebraic decomposition.} Second GI Conference 
on Automata Theory and Formal Languages. LNCS vol 33, 134--183, 
Springer-Verlag, Berlin (1975).

\bibitem{CR} Coste M. and Roy M.-F. {\em Thom's lemma, the coding of real 
algebraic numbers and the computation of the topology 
of semi-algebraic sets.} J. Symbolic Computation, 121--129 (1988).


\bibitem{CLO} Cox D., Little, J. and O'Shea, D. 
{\em Ideals, varieties, and algorithms. 
An introduction to computational algebraic geometry and commutative algebra}, 
Undergraduate Texts in Mathematics, third edition.  Springer, New York, 2007. 


\bibitem{DH} Davenport J.~H. and Heintz J. {\em Real
quantifier elimination is doubly exponential}. {Journal of
Symbolic Computation} 5, 29--35 (1988).


\bibitem{Gri88} Grigoriev D. {}{\em Complexity of deciding
Tarski algebra}. {}{J. Symbolic Comput.} 5, 65--108
(1988).

\bibitem{GV}  Grigoriev D. and Vorobjov N.
{}{\it Solving systems of polynomial inequalities in subexponential time}.
{}{J. Symbolic Comput.} 5, 37--64 (1988).


\bibitem{Hor} H\"ormander L. {\em The analysis of linear partial 
differential operators}, Berlin, Heidelberg, 364--367, Springer, New-York, 
1983. 


\bibitem{Loj} Lojasiewicz S. {\it Ensembles
semi-analytiques.} Institut des Hautes Etudes Scientifiques, 1965.

\bibitem{zerelem} Lombardi H., Perrucci D. and Roy M.-F. {\em An elementary recursive bound for effective Positivstellensatz and
Hilbert 17-th problem}. Manuscript, 2014.

\bibitem{Mo} Monk L.~G.
{\em Elementary-recursive decision procedures}. PhD
thesis, UC Berkeley, 1975.

\bibitem{R92} Renegar  J. {\em On the computational
complexity and geometry of the first-order theory of the reals. I-III.}
{}{J. Symbolic Comput.} 13, 255--352 (1992).


\bibitem{Tarski51} Tarski A. {\em A decision
method for elementary algebra and geometry}, Second Edition. University of
California Press, Berkeley and Los Angeles, California, 1951. 


\bibitem{Seidenberg54} Seidenberg A. {\em A new decision
method for elementary algebra}. {}{Annals of Mathematics}
60, 365--374 (1954).



\end{thebibliography}
\end{document}